\documentclass[twoside,10pt,leqno, A4]{amsart}
\usepackage{amsfonts}
\usepackage{amsmath}
\usepackage{amscd}
\usepackage{amssymb}
\usepackage{amsthm}
\usepackage{amsrefs}
\usepackage{latexsym}
\usepackage{mathrsfs}
\usepackage{bbm}
\usepackage{amscd}
\usepackage{amssymb}
\usepackage{amsthm}
\usepackage{amsrefs}
\usepackage{latexsym}
\usepackage{mathrsfs}
\usepackage{bbm}
\usepackage{enumerate}
\usepackage{graphicx}
\usepackage{color}
\usepackage{mathtools}
\begin{document}
\newtheorem{theorem}[subsection]{Theorem}
\newtheorem{proposition}[subsection]{Proposition}
\newtheorem{lemma}[subsection]{Lemma}
\newtheorem{corollary}[subsection]{Corollary}
\newtheorem{conjecture}[subsection]{Conjecture}
\newtheorem{prop}[subsection]{Proposition}
\newtheorem*{remark}{Remark}
\numberwithin{equation}{section}
\renewcommand{\thefootnote}{\fnsymbol{footnote}}
\newcommand{\dif}{\mathrm{d}}
\newcommand\psum{\mathop{\sum\nolimits^{\mathrlap{*}}}}
\newcommand\dsum{\mathop{\sum\nolimits^{\mathrlap{'}}}}
\newcommand\lle{\mathop{\:\ll\:}_\varepsilon}

\title{A note on the zeros of $L$-functions associated to fixed-order Dirichlet characters}

\author[C.~C.~Corrigan]{C.~C.~Corrigan} 
\address{School of Mathematics and Statistics, University of New South Wales, Sydney, NSW 2052, Australia}
\email{c.corrigan@student.unsw.edu.au}

\date{October, 2023}
\maketitle

\begin{abstract}
    In this paper, we use the Weyl-bound for Dirichlet $L$-functions to derive zero-density estimates for $L$-functions associated to families of fixed-order Dirichlet characters.  The results improve on previous bounds given by the author when $\sigma$ is sufficiently distanced from the critical line.
\end{abstract}

\section{Introduction}
    The distribution of the non-trivial zeros of Dirichlet $L$-functions is of great importance in analytic number theory.   The generalised Riemann hypothesis claims that these zeros all lie on the line $\tfrac{1}{2}+i\textbf{R}$, however, it is possible that they could lie anywhere in the region $(0,1)+i\textbf{R}$.  In practice, it often suffices to employ an unconditional result, known as a zero-density estimate, whose statement is not as strong as the generalised Riemann hypothesis.  In this paper, we let $\chi$ be a primitive Dirichlet character modulo $q$, and suppose that $\sigma\in\big(\tfrac{1}{2},1\big)$ and $T\in(2,\infty)$.  Define $R(\sigma,T)=[\sigma,1]+i[-T,T]$, and let $L(s,\chi)$ be the $L$-function associated to the character $\chi$.  Zero-density estimates are concerned with the number
    \begin{equation*}
        N(\sigma,T,\chi)=\#\{\varrho\in R(\sigma,T) : L(\varrho,\chi)=0\},
    \end{equation*}
    or rather, the average of this number over a family $\mathcal{F}$ of characters.  In this paper, the families $\mathcal{F}=\mathcal{O}_r$ of primitive Dirichlet characters of order $r$ will be the primary interest.  We denote by $\mathcal{O}_r(Q)$ the set of $\chi\in\mathcal{O}_r$ with conductor $q\in(Q,2Q]$.  Note that the generalised Riemann-von Mangoldt formula \cite{rvm1,rvm2} gives the trivial bound $N(\sigma,T,\chi)\ll T\log qT$ for any primitive Dirichlet character $\chi$, which is known to be sharp only for $\sigma=\tfrac{1}{2}$ (cf.\:\cite{bohr,selberg}).\newline
    
    The earliest zero-density estimates to feature an average over a family of Dirichlet $L$-functions are due to Bombieri \cite{bombls}, Vinogradov \cite{vinogradov}, and Montgomery \cite{HM,HM2}.  Results containing averages over $\mathcal{O}_2(Q)$ were first given by Jutila \cite{jut2} and Heath-Brown \cite{DRHB}, both of whom followed the method laid down by Montgomery to derive his results.  The method of Montgomery reduces the problem to estimating mean-values of the type
    \begin{equation*}
        \mathfrak{S}_k(Q,T)=\sum_{\chi\in\mathcal{F}(Q)}\int\limits_{-T}^T\Big|\dsum_{n\leqslant N}a_n\chi(n)n^{-it}\Big|^{2k}\:\dif t
    \end{equation*}
    and
    \begin{equation*}
        \mathfrak{L}_k(Q,T)=\sum_{\chi\in\mathcal{F}(Q)}\int\limits_{-T}^T\big|L\big(\tfrac{1}{2}+it,\chi\big)\big|^{2k}\:\dif t,
    \end{equation*}
    where $\mathcal{F}$ is the character family of interest, and the $'$ denotes that the sum is to be taken over square-free $n$.  A detailed outline of the method has been given in \cite{corrigan}.  In particular, if $\mathcal{F}=\mathcal{O}_r$ for some $r\geqslant2$, we can show that
    \begin{equation}\label{l2est}
        \mathfrak{L}_1(Q,T)\lle(QT)^{1+\varepsilon}\quad\text{when}\quad T^{2r-1}\gg Q^{2r-5}.
    \end{equation}
    Indeed, \eqref{l2est} was proven for the case $r=2$ in \cite{jut1}, the cases $r=3,4,6$ in \cite{msv}, and can be proven in the remaining cases using Theorem~1.6 of \cite{rordersieve}.  Now, the aforementioned zero-density estimate of Jutila is derived using \eqref{l2est} and a sub-optimal bound on $\mathfrak{S}_1(Q,T)$.  In \cite{corrigan,cozh}, we used the large sieve for real characters of Heath-Brown \cite{DRHB} to derive an estimate for $\mathfrak{S}_1(Q,T)$ sharper than those used by Jutila and Heath-Brown to derive their results, and consequently strengthened the estimate of Jutila to
    \begin{equation}\label{quadratic2}
        \sum_{\chi\in\mathcal{O}_2(Q)}N(\sigma,T,\chi)\lle(QT)^\varepsilon\min\Big((QT)^{(4-4\sigma)/(3-2\sigma)},\big(Q^4T^3\big)^{1-\sigma}\Big)
    \end{equation}
    in \cite{msv}.  Additionally, as analogs to \eqref{quadratic2}, we showed for $T^5\gg Q$ that
    \begin{equation}\label{msv3}
        \sum_{\chi\in\mathcal{O}_3(Q)}N(\sigma,T,\chi)\lle(QT)^\varepsilon\min\Big(Q^{(16-10\sigma)/9}T^{(4-4\sigma)/(3-2\sigma)},Q^{(16-16\sigma)/(9-6\sigma)}T^{(4-4\sigma)/(3-2\sigma)},\big(Q^4T^3\big)^{1-\sigma}\Big)
    \end{equation}
    and
    \begin{equation}\label{msv4}
        \sum_{\chi\in\mathcal{O}_4(Q)}N(\sigma,T,\chi)\lle(QT)^\varepsilon\min\Big(Q^{(5-3\sigma)/3}T^{(4-4\sigma)/(3-2\sigma)},Q^{(5-5\sigma)/(3-2\sigma)}T^{(4-4\sigma)/(3-2\sigma)},\big(Q^4T^3\big)^{1-\sigma}\Big).
    \end{equation}
    These results are derived using the large sieve for $\mathcal{O}_3$ and $\mathcal{O}_4$, respectively.  Recently, Balestrieri and Rome \cite{rordersieve} generalised the work of Baier and Young \cite{baieryoung} and Gao and Zhao \cite{gaozhao} to derive a large sieve estimate for general $\mathcal{O}_r$ where $r\geqslant2$.  Using \eqref{l2est} together with Theorem~1.6 of \cite{rordersieve}, we can show along similar lines to \eqref{msv3} and \eqref{msv4} that
    \begin{equation}\label{msvr}
        \sum_{\chi\in\mathcal{O}_r(Q)}N(\sigma,T,\chi)\lle(QT)^\varepsilon\min\Big(Q^{(6-4\sigma)/3}T^{(4-4\sigma)/(3-2\sigma)},Q^{(6-6\sigma)/(3-2\sigma)}T^{(4-4\sigma)/(3-2\sigma)},\big(Q^4T^3\big)^{1-\sigma}\Big)
    \end{equation}
    for $T^{2r-1}\gg Q^{2r-5}$.\newline
    
    Now, if we could demonstrate that the estimate \eqref{l2est} still holds when $k=2$, then we could unconditionally improve on the above zero-density estimates.  Following the method of Heath-Brown \cite{DRHB}, we can at least show for $\mathcal{F}=\mathcal{O}_r$ that
    \begin{equation}\label{l4est}
        \mathfrak{L}_2(Q,T)\lle Q^{1+\varepsilon}T^{2+\varepsilon}\quad\text{when}\quad T\gg Q,
    \end{equation}
    which is sharp in the $Q$-aspect, but unfortunately not in the $T$-aspect.  Nonetheless, we can use \eqref{l4est} to improve on \eqref{quadratic2}, \eqref{msv3}, \eqref{msv4}, and \eqref{msvr} in the $Q$-aspect.  Indeed, in \cite{cozh,corrigan} we used \eqref{l4est} to show that
    \begin{equation}\label{quadratic1}
        \sum_{\chi\in\mathcal{O}_2(Q)}N(\sigma,T,\chi)\lle(QT)^\varepsilon\min\Big(\big(Q^3T^4\big)^{(1-\sigma)/(2-\sigma)},(QT)^{(3-3\sigma)/\sigma}\Big).
    \end{equation}
    Our results in \cite{cozh,corrigan} pertaining to $\mathcal{O}_3$ and $\mathcal{O}_4$ can be improved by showing that Theorem~2.2 therein still holds under a weaker assumption on the relevant large sieve inequality, as in Lemma~\ref{montgomerylemma1} below.  Indeed, under this weaker assumption, we can strengthen the estimates in \cite{cozh} to    
    \begin{equation}\label{msv33}
        \sum_{\chi\in\mathcal{O}_3(Q)}N(\sigma,T,\chi)\lle(QT)^\varepsilon\min\Big(Q^{(16-10\sigma)/9}T^{(4-4\sigma)/(2-\sigma)},Q^{(13-13\sigma)/(6-3\sigma)}T^{(4-4\sigma)/(2-\sigma)},(QT)^{(3-3\sigma)/\sigma}\Big)
    \end{equation}
    for $T^3\gg Q^2$, and
    \begin{equation}\label{msv44}
        \sum_{\chi\in\mathcal{O}_4(Q)}N(\sigma,T,\chi)\lle(QT)^\varepsilon\min\Big(Q^{(5-3\sigma)/3}T^{(4-4\sigma)/(2-\sigma)},(QT)^{(4-4\sigma)/(2-\sigma)},(QT)^{(3-3\sigma)/\sigma}\Big),
    \end{equation}
    for $T^2\gg Q$.  Again, using Theorem~1.6 of \cite{rordersieve}, we can add to these the result
    \begin{equation}\label{msvrr}
        \sum_{\chi\in\mathcal{O}_r(Q)}N(\sigma,T,\chi)\lle(QT)^\varepsilon\min\Big(Q^{(6-4\sigma)/3}T^{(4-4\sigma)/(2-\sigma)},Q^{(5-5\sigma)/(2-\sigma)}T^{(4-4\sigma)/(2-\sigma)},(QT)^{(3-3\sigma)/\sigma}\Big),
    \end{equation}
    which is valid for $T\gg Q$.  These estimates all improve in the $Q$-aspect on their corresponding estimates obtained using \eqref{l2est}.  In general, a sharp bound for $\mathfrak{L}_{k+1}(Q,T)$ will always lead to a stronger zero-density result than a sharp bound for $\mathfrak{L}_k(Q,T)$.  To derive our main results, we consider bounds on $\mathfrak{L}_k(Q,T)$ when $k$ is arbitrarily large.
  
\section{Statement of Results}    
    Unfortunately, no sharp upper bounds have been established for $\mathfrak{L}_k(Q,T)$ when $k>2$.  In \cite{msv} we adapted the method of Heath-Brown \cite{DRHB} to show that $\mathfrak{L}_k(Q,T)\ll_\varepsilon(QT)^{k/2+\varepsilon}$, though we can obtain a better result from a much more trivial approach. Indeed, Petrow and Young \cite{weylbound1,weylbound2} showed that the Weyl-bound
    \begin{equation}\label{newlabel}
        L\big(\tfrac{1}{2}+it,\chi\big)\lle q^{1/6+\varepsilon}(|t|+1)^{1/6+\varepsilon}
    \end{equation}
    holds for any Dirichlet character $\chi$ modulo $q$, from which we derive the trivial bound
    \begin{equation}\label{weylmoment}
        \mathfrak{L}_k(Q,T)\lle(QT)^{k/3+1+\varepsilon}\quad\text{for all}\quad k\geqslant1.
    \end{equation}
    Using \eqref{weylmoment} to estimate $\mathfrak{L}_k(Q,T)$ for arbitrarily large $k$ in the method of Montgomery, we derive the following.    
    \begin{theorem}\label{thm2}
        For any $Q,T\geqslant2$, we have
        \begin{equation*}
            \sum_{\chi\in\mathcal{O}_2(Q)}N(\sigma,T,\chi)\lle(QT)^\varepsilon\min\Big((QT)^{(8-8\sigma)/3},\big(Q^8T^5\big)^{(1-\sigma)/(6\sigma-3)}\Big),
        \end{equation*}
        where $\sigma\in\big(\tfrac{1}{2},1\big)$.
    \end{theorem}
    The above is stronger than \eqref{quadratic1} in the $Q$-aspect when $\sigma>\tfrac{7}{8}$, and stronger in the $T$-aspect for all $\sigma>\tfrac{1}{2}$.  Additionally, it improves on \eqref{quadratic2} precisely when $\sigma>\tfrac{3}{4}$.  Using the same method as is used to prove Theorem~\ref{thm2}, we can show that the density conjecture
    \begin{equation}\label{lindelof}
        \sum_{\chi\in\mathcal{O}_2(Q)}N(\sigma,T,\chi)\lle(QT)^{2(1-\sigma)+\varepsilon}
    \end{equation}
    is a consequence of the Lindel\"of hypothesis holding for $L$-functions of real characters.\newline
    
    For cubic characters, we have the following analog to Theorem~\ref{thm2}.
    \begin{theorem}\label{thm3}
        For any $Q,T\geqslant2$, we have
        \begin{equation*}
            \sum_{\chi\in\mathcal{O}_3(Q)}N(\sigma,T,\chi)\lle(QT)^\varepsilon\min\Big(Q^{(22-16\sigma)/9}T^{(8-8\sigma)/3},Q^{4-4\sigma}T^{(8-8\sigma)/3},\big(Q^8T^5\big)^{(1-\sigma)/(6\sigma-3)}\Big),
        \end{equation*}
        where $\sigma\in\big(\tfrac{1}{2},1\big)$.  Furthermore, the above result holds if $\mathcal{O}_3$ is replaced by $\mathcal{O}_6$.
    \end{theorem}
    The immediate advantage of the above result, as compared to \eqref{msv3} and \eqref{msv33} is that there is no restriction on the relation between $Q$ and $T$.  Additionally, Theorem~\ref{thm3} is stronger than \eqref{msv3} in the $Q$-aspect when $\sigma>\tfrac{5}{6}$ and in the $T$-aspect when $\sigma>\tfrac{3}{4}$, and stronger than \eqref{msv33} in the $Q$-aspect when $\sigma>\tfrac{9}{10}$ and in the $T$-aspect for all $\sigma>\tfrac{1}{2}$.  The fact that Theorem~\ref{thm3} holds for $\mathcal{O}_6$ as well as $\mathcal{O}_3$ is a direct result of Theorem~1.5 of \cite{baieryoung}.  For $\mathcal{O}_4$, we derive the following slightly stronger result.
    \begin{theorem}\label{thm4}
        For any $Q,T\geqslant2$, we have
        \begin{equation*}
            \sum_{\chi\in\mathcal{O}_4(Q)}N(\sigma,T,\chi)\lle(QT)^\varepsilon\min\Big(Q^{(7-5\sigma)/3}T^{(8-8\sigma)/3},Q^{(11-11\sigma)/3}T^{(8-8\sigma)/3},\big(Q^8T^5\big)^{(1-\sigma)/(6\sigma-3)}\Big),
        \end{equation*}
        where $\sigma\in\big(\tfrac{1}{2},1\big)$.
    \end{theorem}
    This improves the $Q$-aspect of \eqref{msv4} whenever $\sigma>\tfrac{9}{11}$ and \eqref{msv44} whenever $\sigma>\tfrac{9}{10}$.  In the $T$-aspect, improvements are made in the same regions as with Theorem~\ref{thm2} and Theorem~\ref{thm3}.  Similarly, we have the following.
    \begin{theorem}\label{thmr}
        For any $Q,T\geqslant2$ and any integer $r\geqslant2$, we have
        \begin{equation*}
            \sum_{\chi\in\mathcal{O}_r(Q)}N(\sigma,T,\chi)\lle(QT)^\varepsilon\min\Big(Q^{(8-6\sigma)/3}T^{(8-8\sigma)/3},Q^{(14-14\sigma)/3}T^{(8-8\sigma)/3},\big(Q^8T^5\big)^{(1-\sigma)/(6\sigma-3)}\Big),
        \end{equation*}
        where $\sigma\in\big(\tfrac{1}{2},1\big)$.
    \end{theorem}
    In the $Q$-aspect, the above improves on \eqref{msvr} when $\sigma>\tfrac{5}{6}$ and on \eqref{msvrr} when $\sigma>\tfrac{9}{10}$.
    \begin{remark}
        The reason for Theorem~\ref{thm2} being the minimum of two quantities, as opposed to Theorems~\ref{thm3},~\ref{thm4},~and~\ref{thmr} being the minima of three quantities, essentially comes down to the fact that the available large sieve inequality for $\mathcal{O}_2$ is optimal, whereas this is not the case for $\mathcal{O}_r$ when $r>2$.  The available large sieve inequalities for $\mathcal{O}_r$ with $r>2$ are given as minima of four quantities, two of which are used in the proofs of the latter three theorems above.  Note, however, that the last term in the minima of the above results is derived using the large-moduli approach of Montgomery (cf.\:Theorem~3.1.3 of \cite{corrigan}), and as a result is independent on the large sieve inequalities (cf.\:\eqref{newnewlabel} below).
    \end{remark}

\section{Lemmata}
    In this section, we present the prerequisites in terms of an arbitrary family $\mathcal{F}$ of primitive Dirichlet characters.  To estimate $\mathfrak{S}_1(Q,T)$, we consider the polynomials $\Delta(Q,T,N)$ such that
    \begin{equation*}
        \mathfrak{S}_1(Q,T)\lle(QN)^\varepsilon\Delta(Q,T,N)\dsum_{n\leqslant N}|a_n|^2
    \end{equation*}
    for all $Q,T,N\geqslant2$ and any sequence $(a_n)_{n\leqslant N}$ of complex numbers.  In practice, a bound for $\Delta(Q,T,N)$ can easily be obtained from the corresponding large sieve estimate, as in \cite{corrigan}.  The method of Montgomery can then be summarised in the following two results.
    \begin{lemma}\label{montgomerylemma1}
        Suppose that $X,Y\geqslant2$ are such that $X\ll Y\ll (QT)^A$ for some absolute constant $A$.  Then  
        \begin{equation*}
            \sum_{\chi\in\mathcal{F}(Q)}N(\sigma,T,\chi)\lle(QT)^\varepsilon\Big(\big(\mathfrak{L}_k(Q,T)\Delta(Q,T,X)^kY^{k(1-2\sigma)}\big)^{1/(k+1)}+\Delta(Q,T,X)X^{1-2\sigma}+\Delta(Q,T,Y)Y^{1-2\sigma}\Big)         
        \end{equation*}
        for any $k\geqslant1$, where the implied constant does not depend on $k$.
    \end{lemma}
    \begin{proof}
        We demonstrated the case $k=1$ in \cite{msv}.  The remaining cases follow similarly, by using H\"older's inequality to derive the estimate
        \begin{equation*}
            \#\mathcal{R}_2\lle(QT)^\varepsilon Y^{k(1-2\sigma)/(k+1)}\Big(\sum_{(\varrho,\chi)\in\mathcal{R}_2}\big|M_X\big(\tfrac{1}{2}+it_\varrho,\chi\big)\big|^2\Big)^{k/(k+1)}\Big(\sum_{(\varrho,\chi)\in\mathcal{R}_2}\big|L\big(\tfrac{1}{2}+it_\varrho,\chi\big)\big|^{2k}\Big)^{1/(k+1)},
        \end{equation*}
        where $\mathcal{R}_2$ is as defined in \cite{msv}.
    \end{proof}
    \begin{lemma}\label{montgomerylemma2}
        For any $Q,T\geqslant2$, we have
        \begin{equation*}
            \sum_{\chi\in\mathcal{F}(Q)}N(\sigma,T,\chi)\lle(QT)^\varepsilon\Big(\big(\mathfrak{L}_k(Q,T)^2Q^{2k}T^k\big)^{(1-\sigma)/(2-k+\sigma(2k-2))}+\big(Q^2T\big)^{(1-\sigma)/(2\sigma-1)}\Big)
        \end{equation*}
        for any $k\geqslant1$, where the implied constant does not depend on $k$.
    \end{lemma}
    \begin{proof}
        The case $k=1$ is shown in \cite{msv}.  The remaining cases follow similarly using the estimate
        \begin{equation*}
            \#\big\{(\varrho,\chi)\in\mathcal{R}_2:\big|L\big(\tfrac{1}{2}+it_\varrho,\chi\big)\big|\geqslant V\big\}\lle (QT)^\varepsilon V^{-2k}\mathfrak{L}_k(Q,T)
        \end{equation*}
        to derive a bound for $\#\mathcal{R}_2$, where $\mathcal{R}_2$ is defined as in \cite{msv}.
    \end{proof}
    In this paper, we consider the case where $k$ is taken arbitrarily large.  The above lemmata are used to derive the following two results, from which our main results follow.
    \begin{lemma}\label{montgomery11}
        Suppose that $\eta,\vartheta\geqslant0$ are constants such that the bound
        \begin{equation*}
            L\big(\tfrac{1}{2}+it,\chi\big)\lle q^{\eta+\varepsilon}(|t|+1)^{\vartheta+\varepsilon}
        \end{equation*}
        holds for all $\chi\in\mathcal{F}$, where $q$ is the conductor of $\chi$.  Then for any $Q,T\geqslant2$, we have
        \begin{equation*}
            \sum_{\chi\in\mathcal{F}(Q)}N(\sigma,T,\chi)\lle(QT)^\varepsilon\Big(Q^{2\eta} T^{2\vartheta}\Delta(Q,T,X)Y^{1-2\sigma}+\Delta(Q,T,X)X^{1-2\sigma}+\Delta(Q,T,Y)Y^{1-2\sigma}\Big),
        \end{equation*}
        where $X,Y\geqslant2$ are as in Lemma~\ref{montgomerylemma1}.
    \end{lemma}
    \begin{proof}
        Using the trivial bound $\#\mathcal{F}(Q)\ll Q^2$ and integrating trivially over $t\in[-T,T]$, the hypothesis gives
        \begin{equation*}
            \mathfrak{L}_k(Q,T)^{1/(k+1)}\lle Q^{(2+2k\eta)/(k+1)+\varepsilon}T^{(1+2k\vartheta)/(k+1)+\varepsilon}
        \end{equation*}
        for any integer $k\geqslant1$.  Consequently,
        \begin{equation*}
            \mathfrak{L}_k(Q,T)^{1/(k+1)}\Delta(Q,T,X)^{k/(k+1)}Y^{k(1-2\sigma)/(k+1)}\lle (QT)^{(A+2)/(k+1)+\varepsilon}Q^{2\eta} T^{2\vartheta}\Delta(Q,T,X)Y^{1-2\sigma},
        \end{equation*}
        where $A$ is as in Lemma~\ref{montgomerylemma1}.  As the implied constant does not depend on $k$, we may take $k$ to be sufficiently large that $(A+2)/(k+1)\leqslant\varepsilon$.  The result then follows by Lemma~\ref{montgomerylemma1}.
    \end{proof}
    \begin{lemma}\label{montgomery22}
        Let $\eta,\vartheta\geqslant0$ be as in Lemma~\ref{montgomery11}, and suppose that $Q,T\geqslant2$.  Then
        \begin{equation*}
             \sum_{\chi\in\mathcal{F}(Q)}N(\sigma,T,\chi)\lle\big(Q^{2+4\eta}T^{1+4\vartheta}\big)^{(1-\sigma)/(2\sigma-1)+\varepsilon}
        \end{equation*}
        whenever $\sigma\geqslant\tfrac{1}{2}+\varepsilon$.
    \end{lemma}
    \begin{proof}
        The result follows from Lemma~\ref{montgomerylemma2} in much the same manner as Lemma~\ref{montgomery11} from Lemma~\ref{montgomerylemma1}.
    \end{proof}

\section{Proof of the Main Results}
    To derive our main results from the above lemmata, we will employ the Weyl-bound \eqref{newlabel}.  Note that the last term in the minima of the theorems follows by taking $(\eta,\vartheta)=\big(\tfrac{1}{6},\tfrac{1}{6}\big)$ in Lemma~\ref{montgomery22}, and thus it suffices to use Lemma~\ref{montgomery11} to prove the remaining terms.
    \begin{proof}[Proof of Theorem~\ref{thm2}]
        As in \cite{cozh,msv}, we can deduce by Corollary~1 of \cite{DRHB} that
        \begin{equation*}
            \Delta(Q,T,N)\ll QT+N.
        \end{equation*}
        For appropriate $\eta,\vartheta\geqslant0$, Lemma~\ref{montgomery11} then gives
        \begin{align}\label{realproof}
            \sum_{\chi\in\mathcal{O}_2(Q)}N(\sigma,T,\chi)&\lle(QT)^\varepsilon\big(Q^{2\eta} T^{2\vartheta}(QT+X)Y^{1-2\sigma}+QTX^{1-2\sigma}+Y^{2-2\sigma}\big)\\
            &\lle Q^{(2+4\eta)(1-\sigma)+\varepsilon}T^{(2+4\vartheta)(1-\sigma)+\varepsilon}\notag
        \end{align}
        on taking $X=QT$ and $Y=Q^{1+2\eta}T^{1+2\vartheta}$, from which the assertion follows on taking $(\eta,\vartheta)=\big(\tfrac{1}{6},\tfrac{1}{6}\big)$.
    \end{proof}
    \begin{proof}[Proof of Theorem~\ref{thm3}]
        As in \cite{cozh,msv}, we can show using Theorem~1.4 of \cite{baieryoung} that
        \begin{equation*}
            \Delta(Q,T,N)\ll\min\big(Q^{5/3}T+N,Q^{11/9}T+Q^{2/3}N\big).
        \end{equation*}
        By Lemma~\ref{montgomery11}, we see for appropriate $\eta,\vartheta\geqslant0$ that
        \begin{align}\label{cubicproof}
            \sum_{\chi\in\mathcal{O}_3(Q)}N(\sigma,T,\chi)&\lle(QT)^\varepsilon\big(Q^{2\eta} T^{2\vartheta}\min(Q^{5/3}T+X,Q^{11/9}T+Q^{2/3}X\big)Y^{1-2\sigma}\\            
            &\hspace{20mm}+\min\big(Q^{5/3}TX^{1-2\sigma}+Y^{2-2\sigma}, Q^{11/9}TX^{1-2\sigma}+Q^{2/3}Y^{2-2\sigma}\big)\notag\\
            &\lle\min\big(Q^{(10/3+4\eta)(1-\sigma)+\varepsilon},Q^{2/3+(10/9+4\eta)(1-\sigma)+\varepsilon}\big)T^{(2+4\vartheta)(1-\sigma)+\varepsilon},\notag
        \end{align}
        where in the first term of the minimum we have taken
        \begin{equation*}
            X=Q^{5/3}T \quad \text{and} \quad Y=Q^{5/3+2\eta}T^{1+2\vartheta},
        \end{equation*}
        and in the second we have taken
        \begin{equation*}
            X=Q^{5/9}T \quad \text{and} \quad Y=Q^{5/9+2\eta}T^{1+2\vartheta}.
        \end{equation*}
        The desired result follows from taking $(\eta,\vartheta)=\big(\tfrac{1}{6},\tfrac{1}{6}\big)$ in \eqref{cubicproof}.
    \end{proof}
    \begin{proof}[Proof of Theorem~\ref{thm4}]
        As in \cite{msv,cozh}, Lemma~2.10 of \cite{gaozhao} can be used to show that
        \begin{equation*}
            \Delta(Q,T,N)\ll\min\big(Q^{3/2}T+N,Q^{7/6}T+Q^{2/3}N\big).
        \end{equation*}
        Then, by Lemma~\ref{montgomery11}, for appropriate $\eta,\vartheta\geqslant0$ we have
        \begin{align}\label{quarticproof}
            \sum_{\chi\in\mathcal{O}_4(Q)}N(\sigma,T,\chi)&\lle(QT)^\varepsilon\big(Q^{2\eta} T^{2\vartheta}\min(Q^{3/2}T+X,Q^{7/6}T+Q^{2/3}X\big)Y^{1-2\sigma}\\            
            &\hspace{20mm}+\min\big(Q^{3/2}TX^{1-2\sigma}+Y^{2-2\sigma}, Q^{7/6}TX^{1-2\sigma}+Q^{2/3}Y^{2-2\sigma}\big)\notag\\
            &\lle\min\big(Q^{(3+4\eta)(1-\sigma)+\varepsilon},Q^{2/3+(1+4\eta)(1-\sigma)+\varepsilon}\big)T^{(2+4\vartheta)(1-\sigma)+\varepsilon},\notag
        \end{align}
        where in the first term of the minimum we have taken
        \begin{equation*}
            X=Q^{3/2}T \quad \text{and} \quad Y=Q^{3/2+2\eta}T^{1+2\vartheta},
        \end{equation*}
        and in the second we have taken
        \begin{equation*}
            X=Q^{1/2}T \quad \text{and} \quad Y=Q^{1/2+2\eta}T^{1+2\vartheta}.
        \end{equation*}
        The assertion then follows from taking $(\eta,\vartheta)=\big(\tfrac{1}{6},\tfrac{1}{6}\big)$ in \eqref{quarticproof}.
    \end{proof}
    \begin{proof}[Proof of Theorem~\ref{thmr}]
        It follows from Theorem~1.6 of \cite{rordersieve} that
        \begin{equation*}
            \Delta(Q,T,N)\ll\min\big(Q^2T+N,Q^{4/3}T+Q^{2/3}N\big).
        \end{equation*}
        Then, by Lemma~\ref{montgomery11}, for appropriate $\eta,\vartheta\geqslant0$ we have
        \begin{align}\label{rproof}
            \sum_{\chi\in\mathcal{O}_4(Q)}N(\sigma,T,\chi)&\lle(QT)^\varepsilon\big(Q^{2\eta} T^{2\vartheta}\min(Q^2T+X,Q^{4/3}T+Q^{2/3}X\big)Y^{1-2\sigma}\\            
            &\hspace{20mm}+\min\big(Q^2TX^{1-2\sigma}+Y^{2-2\sigma}, Q^{4/3}TX^{1-2\sigma}+Q^{2/3}Y^{2-2\sigma}\big)\notag\\
            &\lle\min\big(Q^{(4+4\eta)(1-\sigma)+\varepsilon},Q^{2/3+(4/3+4\eta)(1-\sigma)+\varepsilon}\big)T^{(2+4\vartheta)(1-\sigma)+\varepsilon},\notag
        \end{align}
        where in the first term of the minimum we have taken
        \begin{equation*}
            X=Q^2T \quad \text{and} \quad Y=Q^{2+2\eta}T^{1+2\vartheta},
        \end{equation*}
        and in the second we have taken
        \begin{equation*}
            X=Q^{2/3}T \quad \text{and} \quad Y=Q^{2/3+2\eta}T^{1+2\vartheta}.
        \end{equation*}
        The proof is complete on taking $(\eta,\vartheta)=\big(\tfrac{1}{6},\tfrac{1}{6}\big)$ in \eqref{rproof}.
    \end{proof}
    It is clear from \eqref{realproof} how the density conjecture for real characters \eqref{lindelof} follows from the Lindel\"of hypothesis.  However, an analogous result of the same strength cannot be established for $\mathcal{O}_3$, $\mathcal{O}_4$, or $\mathcal{O}_r$ using \eqref{cubicproof}, \eqref{quarticproof}, or \eqref{rproof}, respectively.  Additionally, it is clear that the bound derived from Lemma~\ref{montgomery22} has no dependence on the character family $\mathcal{F}$.  Indeed, using the above method, we can show that
    \begin{equation}\label{newnewlabel}
        \sum_{q\leqslant Q}\psum_{\chi\bmod{q}}N(\sigma,T,\chi)\ll(QT)^\varepsilon\min\Big(Q^{(14-14\sigma)/3}T^{(8-8\sigma)/3},Q^{(8-8\sigma)/(6\sigma-3)}T^{(5-5\sigma)/(6\sigma-3)}\Big),
    \end{equation}
    which improves on Theorem~12.2 of \cite{HM}, and confirms the density conjecture for $\sigma>\tfrac{11}{12}$.  Heath-Brown \cite{heathy}, however, was able to show that the density conjecture holds in the larger range $\sigma>\tfrac{11}{14}$.
    
\vspace*{.5cm}

\noindent{\bf Acknowledgments.}  The author would like to thank the University of New South Wales for access to some of the resources that were necessary to complete this paper, as well as the Commonwealth for its support through an Australian Government Research Training Program Scholarship.  Thanks is also given to Dr.~Liangyi~Zhao who brought the article \cite{rordersieve} to the author's attention, and to the referee for their comprehensive review of the original draft.

\bibliography{secondmoment}

\begin{bibdiv}
\begin{biblist}

\bib{baieryoung}{article}{
      author={Baier, S.},
      author={Young, M.~P.},
       title={Mean values with cubic characters},
        date={2010},
     journal={J. Number Theory},
      volume={130},
       pages={879\ndash 903},
}

\bib{rordersieve}{article}{
      author={Balestrieri, F.},
      author={Rome, N.},
       title={Average {B}ateman-{H}orn for {K}ummer polynomials},
        date={2023},
     journal={Acta Arith.},
      volume={207},
       pages={315\ndash 350},
}

\bib{bohr}{article}{
      author={Bohr, H.~A.},
      author={Landau, E. G.~H.},
       title={Ein {S}atz {\"u}ber {D}irichletsche {R}eihen mit {A}nwendung auf die $\zeta$-{F}unktion und die {$L$}-{F}unktionen},
        date={1914},
     journal={Rend. Circ. Mat. Palermo},
      volume={37},
       pages={269\ndash 272},
}

\bib{bombls}{article}{
      author={Bombieri, E.},
       title={On the large sieve},
        date={1965},
     journal={Mathematika},
      volume={12},
       pages={201\ndash 225},
}

\bib{corrigan}{misc}{
      author={Corrigan, C.~C.},
       title={On the distribution of zeros for families of {D}irichlet ${L}$-functions},
 institution={The University of New South Wales},
        date={2022},
        note={(Honours Thesis)},
}

\bib{msv}{article}{
      author={Corrigan, C.~C.},
       title={Mean square values of {D}irichlet {$L$}-functions associated to fixed-order characters},
        date={2024},
     journal={J. Number Theory},
      volume={256},
       pages={8\ndash 22},
}

\bib{cozh}{article}{
      author={Corrigan, C.~C.},
      author={Zhao, L.},
       title={Zero density theorems for families of {D}irichlet ${L}$-functions},
        date={2023},
     journal={Bull. Aust. Math. Soc.},
      volume={108},
       pages={224\ndash 235},
}

\bib{rvm2}{book}{
      author={Davenport, H.},
       title={Multiplicative number theory},
     edition={3},
      series={Graduate Texts in Mathematics, vol. 74},
   publisher={Springer},
        date={2000},
}

\bib{gaozhao}{article}{
      author={Gao, P.},
      author={Zhao, L.},
       title={Moments of central values of quartic {D}irichlet {$L$}-functions},
        date={2021},
     journal={J. Number Theory},
      volume={228},
       pages={342\ndash 358},
}

\bib{heathy}{article}{
      author={Heath-Brown, D.~R.},
       title={The density of zeros of {D}irichlet's {$L$}-functions},
        date={1979},
     journal={Can. J. Math.},
      volume={31},
       pages={231\ndash 240},
}

\bib{DRHB}{article}{
      author={Heath-Brown, D.~R.},
       title={A mean value estimate for real character sums},
        date={1995},
     journal={Acta Arith.},
      volume={72},
       pages={235\ndash 275},
}

\bib{jut2}{article}{
      author={Jutila, M.},
       title={On mean values of {D}irichlet polynomials with real characters},
        date={1975},
     journal={Acta Arith.},
      volume={27},
       pages={191\ndash 198},
}

\bib{jut1}{article}{
      author={Jutila, M.},
       title={On mean values of {$L$}-functions and short character sums with real characters},
        date={1975},
     journal={Acta Arith.},
      volume={26},
       pages={405–410},
}

\bib{HM2}{article}{
      author={Montgomery, H.~L.},
       title={Zeros of {$L$}-functions},
        date={1969},
     journal={Invent. Math.},
      volume={8},
       pages={346–354},
}

\bib{HM}{book}{
      author={Montgomery, H.~L.},
       title={{T}opics in {M}ultiplicative {N}umber {T}heory},
      series={Lecture Notes in Mathematics},
   publisher={Spring-Verlag},
     address={Berlin},
        date={1971},
      volume={227},
}

\bib{weylbound1}{article}{
      author={Petrow, I.},
      author={Young, M.~P.},
       title={The {W}eyl bound for {D}irichlet {$L$}-functions of cube-free conductor},
        date={2020},
     journal={Annals Math.},
      volume={192},
       pages={437\ndash 486},
}

\bib{weylbound2}{article}{
      author={Petrow, I.},
      author={Young, M.~P.},
       title={The fourth moment of {D}irichlet {$L$}-functions along a coset and the {W}eyl bound},
        date={2023},
     journal={Duke Math. J.},
      volume={172},
       pages={1879\ndash 1960},
}

\bib{rvm1}{book}{
      author={Prachar, K.},
       title={Primzahlverteilung},
      series={Grundlehren der mathematischen Wissenschaften, Band 91},
   publisher={Springer},
        date={1957},
}

\bib{selberg}{article}{
      author={Selberg, A.},
       title={On the zeros of {R}iemann’s zeta-function},
        date={1942},
     journal={Skr. Norske Vid. Akad. Oslo I, no. 10},
       pages={1\ndash 59},
}

\bib{vinogradov}{article}{
      author={Vinogradov, A.~I.},
       title={On the density hypothesis for {D}irichlet {$L$}-series},
        date={1965},
     journal={Izv. Akad. Nauk SSSR Ser. Mat.},
      volume={29},
       pages={903–934},
}

\end{biblist}
\end{bibdiv}
\bibliographystyle{amsxport}

\end{document}